\documentclass[12pt,oneside]{article}

\usepackage{latexsym,amssymb,amsfonts,amsmath,amsthm,ifthen}

\theoremstyle{plain}
\newtheorem{theorem}{Theorem}[section]
\newtheorem{lemma}[theorem]{Lemma}
\newtheorem{proposition}[theorem]{Proposition}

\theoremstyle{remark}
\newtheorem*{rems}{Remark}

\DeclareMathOperator{\Dom}{Dom}

\DeclareMathOperator{\Real}{Re}
\DeclareMathOperator{\supp}{supp}
\DeclareMathOperator{\spec}{spec}

\newcommand{\loc}{\mathrm{loc}}
\newcommand{\rd}{\mathrm{d}}
\newcommand{\re}{\mathrm{e}}
\newcommand{\ri}{\mathrm{i}}

\newcommand{\R}{\mathbb{R}}
\newcommand{\N}{\mathbb{N}}

\newcommand{\C}{\mathbb{C}}
\newcommand{\Z}{\mathbb{Z}}
\renewcommand{\epsilon}{\varepsilon}

\newcommand{\ipd}[2]{\langle{#1},{#2}\rangle}
\newcommand{\bigipd}[2]{\bigl\langle{#1},\hspace{0.1em}{#2}\bigr\rangle}
\newcommand{\abs}[1]{\lvert{#1}\rvert}
\newcommand{\bigabs}[1]{\bigl\lvert{#1}\bigr\rvert}
\newcommand{\norm}[1]{\lVert{#1}\rVert}
\newcommand{\bignorm}[1]{\bigl\lVert{#1}\bigr\rVert}
\newcommand{\bk}[1]{\langle{#1}\rangle}

\newcommand{\wt}[1]{\widetilde{#1}}

\newcommand{\Ereg}{\Omega_{R}}
\newcommand{\halfsp}[1][d]{\R^{#1}_+}
\newcommand{\Rlattice}{\Gamma}
\newcommand{\RDlattice}{\Gamma^{\dagger}}
\newcommand{\Rucell}{\mathcal{O}}
\newcommand{\eRucell}{\mathcal{P}}
\newcommand{\RDucell}{\mathcal{O}^{\dagger}}
\newcommand{\tcell}[2]{\mathcal{O}_{#1,#2}}
\newcommand{\etcell}[2]{\mathcal{P}_{#1,#2}}
\newcommand{\bde}{\mathbf{e}}
\newcommand{\bdf}{\mathbf{f}}
\newcommand{\torus}{\mathbb{T}}

\newcommand{\BFt}[1]{\ifthenelse{\equal{#1}{}}{\mathcal U}{(\mathcal U#1)_{\theta}}}

\newcommand{\Sch}[1][V]{\ifthenelse{\equal{#1}{}}{-\Delta}{-\Delta+{#1}}}
\newcommand{\Tlap}[1][\theta]{H_{#1}}
\newcommand{\OpA}{A}
\newcommand{\OpB}[1][t]{B_{#1}}
\newcommand{\dt}{\nabla_t}
\newcommand{\OpH}{\dt^2-\OpA}
\newcommand{\OpN}{N}
\newcommand{\OpAw}{A_{\ftew}}
\newcommand{\OpL}{L}
\newcommand{\OpBM}[1]{\mathbf{M}_{#1}}
\newcommand{\proj}[1]{P_{#1}}
\newcommand{\Qroj}[1]{\mathbf{Q}_{#1}}

\newcommand{\hs}[1][]{X_{#1}}
\newcommand{\HS}[1][]{\mathbf{X}_{#1}}
\newcommand{\fs}[2][\loc]{\ifthenelse{\equal{#1}{\loc}}{\mathcal{X}_{\loc}^{#2}}{\mathcal{X}_{#1}^{#2}}}

\newcommand{\tfn}[3][\theta]{\ifthenelse{\equal{#2}{}}{\phi_{#1}}{\phi_{#1}(#3,#2)}}
\newcommand{\ctf}{h}
\newcommand{\remRphi}{\rho(R)}
\newcommand{\blsup}{\beta} %{\overline{b}}
\newcommand{\Ftew}{\Omega}
\newcommand{\ftew}{\omega}
\newcommand{\bPhi}{\mathbf{\Phi}}
\newcommand{\bPsi}{\mathbf{\Psi}}

\newcommand{\glb}[1][]{\mu_{#1}}
\newcommand{\gub}[1][]{\nu_{#1}}
\newcommand{\glen}[1][]{\gamma_{#1}}
\newcommand{\gcen}[1][]{\lambda_{#1}}

\newcommand{\iqf}{\mathbf{q}}
\newcommand{\Viqf}{Q}

\begin{document}

\title{Decay rates at infinity for solutions to periodic Schr\"{o}dinger equations}
\author{Daniel M.~Elton}

\maketitle

\begin{abstract}
We consider the equation $\Delta u=Vu$ in the half-space $\halfsp$, $d\ge2$
where $V$ has certain periodicity properties. 
In particular we show that such equations cannot have non-trivial 
superexponentially decaying solutions. As an application this leads 
to a new proof for the absolute continuity of the spectrum of particular
periodic Schr\"{o}dinger operators. 
The equation $\Delta u=Vu$ is studied as part of a broader class 
of elliptic evolution equations.
\end{abstract}

\section{Introduction}

We are interested in the possible decay rate of (distributional) solutions to the equation
\begin{equation}
\label{eq:basic}
(\Sch)u=Eu
\end{equation}
where $\Delta$ is the Laplace operator on $\R^d$, $V$ is a measurable function
and $E$ is a constant.
Landis (see \cite{KL}) asked if the boundedness of $V$ is sufficient to exclude 
superexponentially decaying solutions.
More precisely, suppose $V$ is bounded and $u$ solves \eqref{eq:basic} in the exterior region
$\Ereg=\{x\in\R^d:\abs{x}>R\}$, $R>0$, 
while $\re^{\lambda\abs{x}}u$ is bounded on $\Ereg$ for all $\lambda>0$;
does it follow that $u\equiv0$ on $\Ereg$?

Viewing $E$ as a spectral parameter \eqref{eq:basic} is the spectral equation for the Schr\"odinger operator with potential $V$.
In this context Simon (\cite{S}) posed a related question about superexponentially decaying solutions; 
in particular, if $V$ is such that $\Sch$  
defines a self-adjoint operator with a non-compact resolvent does any non-trivial solution of \eqref{eq:basic}
satisfy $\re^{\lambda\abs{x}}u\notin L^2$ for some $\lambda>0$? Note that, $V$ must be real-valued for $\Sch$ to be self-adjoint,
while $\Sch$ has a non-compact resolvent for any bounded $V$.

If one considers complex-valued $V$ the answer to Landis' question is negative. 
In particular, given $\epsilon\in[0,1/2)$ there exists a continuous complex-valued $V$ on $\R^2$ 
and non-trivial $u\in C^2(\R^2)$ with $V(x)=O(\abs{x}^{-\epsilon})$ as $\abs{x}\to\infty$, $\Delta u=Vu$ on $\R^2$, 
and $\re^{\lambda\abs{x}^{(4-2\epsilon)/3}}u\in L^\infty$ for some $\lambda>0$; 
see \cite{M} for the case $\epsilon=0$ and \cite{C-S} for the generalisation to $\epsilon>0$.

On the other hand Landis' question is known to have a positive answer when $d=1$ 
(essentially a classical result for ordinary differential equations), when $d=2$, $E=0$ and $V\ge0$ (\cite{KLW}), 
and for any $E\in\R$ and $V$ with $V(x)=O(\abs{x}^{-1/2})$ as $\abs{x}\to\infty$ (\cite{FHHO1,M2}).
For $E\in\R$ and bounded real-valued $V$, superexponentially decaying solutions of \eqref{eq:basic} 
can also be excluded under some conditions which stabilise $V(x)$ for large $x$; in particular this holds when
(the distributional derivative) $(x.\nabla)V$ is also bounded (\cite{BM,FHHO1}).
A complete answer to Landis' question for real-valued potentials, or to the more general question of Simon,
remains unknown.

In the present work we consider functions $V$ which are periodic transverse to a given direction. 
This naturally favours working on a half-space; 
since $\Ereg$ includes a translated copy of any half-space our results will also apply to exterior regions.
Let $\R_+=(0,+\infty)$ and $\halfsp=\R^{d-1}\times\R_+$. 
For $x\in\halfsp$ we will use the notation $x=(\wt{x},t)$ where $\wt{x}=(x_1,\dots,x_{d-1})\in\R^{d-1}$ and $t=x_d>0$. 
Also set $\bk{\wt{x}}=(1+\abs{\wt{x}}^2)^{1/2}$.
We obtain the following.

\begin{theorem}
\label{thm:mainresper}
Suppose $V\in L^\infty(\halfsp)$ is periodic with respect to a lattice $\Rlattice\subset\R^{d-1}$
and $E\in\R$.
Let $u\in L^2_\loc$ be a (distributional) solution of \eqref{eq:basic} on $\halfsp$ which satisfies
\begin{equation}
\label{eq:hypestmainresper}
\int_{\halfsp}\bk{\wt{x}}^{2\kappa}\re^{2\lambda t}\abs{u(x)}^2\rd\wt{x}\,\rd t<+\infty
\end{equation}
for all $\lambda>0$ and some $\kappa$. Suppose we also have (at least) one of the following:
\begin{itemize}
\item[\upshape (i)]
$d=2$ and $\kappa\ge0$.
\item[\upshape (ii)]
$d=3$, $\Rlattice$ is rational and $\kappa>1$.
\item[\upshape (iii)]
$d\ge2$, $\Rlattice$ is rational, $\kappa>(d-1)/2$ and 
$\norm{V(\cdot,t)}_{L^\infty(\R^{d-1})}\to0$ as $t\to+\infty$.
\end{itemize}
Then $u\equiv0$ on $\halfsp$.  
\end{theorem}

By periodicity of $V$ with respect to $\Rlattice$ we mean $V(\wt{x}+l,t)=V(\wt{x},t)$ 
for any $(\wt{x},t)\in\halfsp$ and $l\in\Rlattice$.

\begin{rems}
Note that $V$ is allowed to be complex-valued.
In cases (i) and (ii) we can absorb $E$ into $V$; 
it follows that we can allow complex $E$ in these cases.
Also note that the conditions on $V$ are satisfied by any potential which is periodic 
with respect to a lattice on $\R^d$, provided this lattice has a rational rank $d-1$ sublattice.
\end{rems}

\begin{rems}
Exponential decay (namely \eqref{eq:hypestmainresper} for \emph{some} $\lambda>0$) is not sufficient.
For example, 
\[
u(x)=\frac{\re^{\ri(x_1+\ri x_2)}}{(x_1+\ri x_2)+\ri}
\]
defines a harmonic function on $\halfsp[2]$ (so $(\Sch)u=0$ with $V=0$) while
$\re^{\lambda x_2}u\in L^2(\halfsp[2])$
for any $\lambda<1$; however $u\not\equiv0$.
\end{rems}

Theorem \ref{thm:mainresper} and related results on the non-existence of solutions with certain types of decay
can be viewed as unique continuation theorems at infinity for \eqref{eq:basic}. 
The implied lower bounds on the decay rate (possibly in a more quantitative form) have applications to spectral 
questions for Schr\"{o}dinger operators (such as the exclusion of embedded eigenvalues; see \cite{FHHO2} for example). 
For periodic potentials an important link was established in \cite[theorem 4.1.5]{K}; in particular, if $V\in L^\infty(\R^d)$ is 
periodic (with respect to a lattice on $\R^d$) then the self-adjoint operator $\Sch$ has an eigenvalue iff \eqref{eq:basic}
has a super\-exponentially decaying solution. 
As a corollary of theorem \ref{thm:mainresper} we thus obtain a new proof 
for the following particular case of a well known result of Thomas (\cite{T}).

\begin{theorem}
Let $d=2$ or $3$ and suppose $V\in L^\infty(\R^d)$ is real-valued and periodic with respect to 
a lattice on $\R^d$ which contains a rational rank $d-1$ sublattice.
Then the spectrum of the self-adjoint operator $\Sch$ contains no eigenvalues.
\end{theorem}

Existing proofs of this result and its many generalisations make use of Bloch (or Floquet) analysis
and the analytic extension of the resulting operators into complex values of the quasi-momentum.

\medskip

Theorem \ref{thm:mainresper} is obtained as a special case of a more general result which we now describe.
Let $\OpA$ be a lower semi-bounded self-adjoint operator on a Hilbert space $\hs$. 
For $j=0,1,2$ set $\hs[j]=\Dom(\abs{\OpA}^{j/2})\subseteq\hs$, so
$\hs[2]=\Dom(\OpA)$, $\hs[1]$ is the form domain of $\OpA$ and $\hs[0]=\hs$.
For $j=1,2$ we can make $\hs[j]$ into a Hilbert space using the isomorphism $\abs{A}^{j/2}+I:\hs[j]\to\hs$. 

Let $\dt$ denote differentiation with respect to $t$.
We want to consider the operator $\OpH$ which maps $\fs{2}\to\fs{0}$ where
\[
\fs{2}=L^2_\loc(\R_+,\hs[2])\cap H^1_\loc(\R_+,\hs[1])\cap H^2_\loc(\R_+,\hs)
\quad\text{and}\quad
\fs{0}=L^2_\loc(\R_+,\hs).
\]
We are interested in the possible decay rate of functions $\phi\in\fs{2}$ which satisfy
$(\OpH)\phi=\OpB\phi$ where $t\mapsto\OpB$ is a uniformly bounded family of 
operators on $\hs$. At this general level we obtain the following.

\begin{theorem}
\label{thm:mainresarbAest}
Let $\phi\in\fs{2}$ and suppose 
$\norm{(\OpH)\phi(t)}_{\hs}\le\beta\norm{\phi(t)}_{\hs}$ for some $\beta>0$ and all $t>0$. 
If $\phi$ satisfies 
\begin{equation}
\label{eq:hypestmainresarbAest}
\int_0^\infty \re^{2\lambda t^{4/3}}\norm{\phi(t)}_{\hs}^2\rd t<+\infty
\end{equation}
for all $\lambda>0$, then we must have $\phi\equiv0$.
\end{theorem}

This result extends \cite[theorem 1]{M}; 
indeed we can recover the latter by taking $A$ to be minus the Laplacian on $\R^{d-1}$.
Furthermore the example constructed in \cite[\S2]{M} shows the decay rate limits given by 
theorem \ref{thm:mainresarbAest} cannot be improved in general.
To exclude any non-trivial solutions with superexponential decay as $t\to+\infty$ we 
impose further conditions on the gaps in the spectrum of $\OpA$, possibly in conjunction 
with some form of decay for $\OpB$.

\begin{theorem}
\label{thm:mainresgen}
Let $\phi\in\fs{2}$ and suppose 
$\norm{(\OpH)\phi(t)}_{\hs}\le b(t)\norm{\phi(t)}_{\hs}$ for some $b\in L^\infty(\R_+)$.
Further suppose one of the following is satisfied:
\begin{itemize}
\item[\upshape (i)]
$\spec(\OpA)$ contains arbitrarily large positive gaps.
\item[\upshape (ii)]
there exists $\delta>0$ such that $\R_+\setminus\spec(\OpA)$ 
contains infinitely many disjoint intervals of length $\delta$, and $b(t)\to0$ as $t\to+\infty$.
\end{itemize}
If $\phi$ satisfies 
\begin{equation}
\label{eq:hypestmainresgen}
\int_0^\infty \re^{2\lambda t}\norm{\phi(t)}_{\hs}^2\rd t<+\infty
\end{equation}
for all $\lambda>0$, then we must have $\phi\equiv0$.
\end{theorem}

The conditions on $\OpA$ are satisfied by a number of standard operators. 
For example, condition (i) holds if $\OpA$ is 
minus the Laplace-Beltrami operator on the $n$-sphere $\mathbb{S}^n$ for $n\ge1$, 
or any positive elliptic pseudo-differential operator of order $m$ on a closed $n$-dimensional manifold 
provided $m>n$; in the latter case the spectral part of condition (ii) is met when $m=n$. 
In order to deduce theorem \ref{thm:mainresper} we need to consider 
the Laplacian on $(d-1)$-dimensional tori; the rationality assumption is then used to establish
the existence of arbitrarily large positive gaps when $d=3$, 
or infinitely many gaps of a uniform size for arbitrary $d$ (see proposition \ref{prop:specgapTlap} below). 
It is not known if arbitrarily large positive gaps exist for all $2$-dimensional tori.

\medskip

We can consider $\OpH$ as an `elliptic evolution operator'. 
The corresponding parabolic and hyperbolic evolution operators, $\dt+\OpA$ and $\dt^2+\OpA$ respectively,
were considered in \cite{M} where results in the spirit of 
theorems \ref{thm:mainresarbAest} and \ref{thm:mainresgen} were obtained.

\medskip

Theorems \ref{thm:mainresarbAest} and \ref{thm:mainresgen} are obtained from Carleman type estimates 
(see propositions \ref{prop:carlest43} and \ref{prop:carlest} respectively) using standard arguments; 
these are presented in \S\ref{sec:genres}, with the proofs of the Carleman estimates being given in \S\ref{sec:carlestpf}.
In \S\ref{sec:perres} we deduce theorem \ref{thm:mainresper} from theorem \ref{thm:mainresgen} and consideration of gaps 
in the spectra of Laplace type operators on tori.

\section{Periodic result}
\label{sec:perres}

Let $\Rlattice\subset\R^{d-1}$ be a lattice with $\RDlattice\subset\R^{d-1}$ denoting the dual lattice. 
Also let $\Rucell$ and $\RDucell$ denote unit cells of $\Rlattice$ and $\RDlattice$ respectively.
Set $\torus=\R^{d-1}/\Rlattice$, the $(d-1)$-dimensional torus corresponding to $\Rlattice$. 
For each $t\in\R_+$ the function $\wt{x}\mapsto V(\wt{x},t)$ is $\Rlattice$-periodic so can be viewed 
as a function $V_t\in L^\infty(\torus)$; 
the mapping $t\mapsto V_t$ is uniformly bounded in $t$.
We will apply a Bloch-Floquet decomposition to $\Sch[]-E$ (see \cite{RS4}); this leads to 
a family of lower semi-bounded self-adjoint elliptic operators on $\torus$ defined by
\[
\Tlap=(-\ri\nabla_{\wt{x}}+\theta)^2-E
\]
for $\theta\in\RDucell$.
The operator $\dt^2-\Tlap$ maps $\fs{2}\to\fs{0}$ where
\[
\fs{2}=\bigcap_{j=0,1,2}H^j_\loc((1,\infty),H^{2-j}(\torus))
\quad\text{and}\quad
\fs{0}=L^2_\loc((1,\infty),L^2(\torus)).
\]
The Bloch-Floquet decomposition is implemented by the Gelfand transform; for 
$v\in L^2(\R^{d-1})$ set
\[
\BFt{v}(\wt{x})=\sum_{l\in\Rlattice} \re^{-\ri\theta.(\wt{x}+l)}v(\wt{x}+l),
\quad\theta\in\RDucell,\,\wt{x}\in\R^{d-1}.
\]
This expression is clearly $\Rlattice$-periodic in $\wt{x}$ so can be viewed as a function on $\torus$;
in fact $\BFt{}$ is a unitary mapping $L^2(\R^{d-1})\to L^2(\RDucell,L^2(\torus))$. 

Let $u\in L^2(\halfsp)$. 
For each $\theta\in\RDucell$ and $t>0$ set $\tfn{t}{\cdot}=\BFt{u(\cdot,t)}$, 
considered as an element of $L^2(\torus)$. 

\begin{lemma}
\label{lem:basicpropphitheta}
Suppose $u$ is a distributional solution of \eqref{eq:basic} 
on $\halfsp$ and 
satisfies \eqref{eq:hypestmainresper} for some $\kappa\ge0$ and $\lambda>0$. 
Then $\tfn{}{}\in\fs{2}$ with $(\dt^2-\Tlap)\tfn{}{}=V_t\tfn{}{}$ and 
\begin{equation}
\label{eq:hypestlemper}
\int_{1}^\infty \re^{2\lambda t}\norm{\tfn{t}{\cdot}}_{L^2(\torus)}^2\rd t<+\infty
\end{equation}
for almost all $\theta\in\RDucell$. 
If $\kappa>(d-1)/2$ then the same conclusion holds for all $\theta\in\RDucell$, 
while $\tfn{}{}$ depends continuously on $\theta$.
\end{lemma}

\begin{proof}
Choose a basis $\{\bde_1,\dots,\bde_{d-1}\}$ for $\Rlattice$ corresponding to the unit cell $\Rucell$.
Set 
\[
\eRucell=\bigl\{r_1\bde_1+\dots+r_{d-1}\bde_{d-1}:r_1,\dots,r_{d-1}\in(-1/2,3/2)\bigr\}
\]
so $\bigcup_{l\in\Rlattice}(\eRucell+l)$ is a $2^{d-1}$-fold covering of $\R^{d-1}$.
For any $l\in\Rlattice$ and $s\ge1$ let $\tcell{l}{s}=(\Rucell+l)\times(s,s+1)$ and $\etcell{l}{s}=(\eRucell+l)\times(s-1,s+2)$.
Note that $\tcell{l}{s}\subset\subset\etcell{l}{s}$. 

From \eqref{eq:hypestmainresper} we get $u\in L^2(\halfsp)$ (recall that $\kappa\ge0$ and $\lambda>0$), 
while $\Delta u=(V-E)u$ as distributions and $V$ is uniformly bounded. It follows that $u\in H^2_\loc(\halfsp)$ and
\begin{equation}
\label{eq:ellregu1}
\norm{u}_{H^2(\tcell{l}{s})}\le C_{1}\norm{u}_{L^2(\etcell{l}{s})}
\end{equation}
for some constant $C_{1}$ which is independent of $l\in\Rlattice$ and $s\ge 1$.
This uniformity leads to $u\in H^2(\R^{d-1}\times(1,\infty))$.
Applying the Gelfand transform we get 
\[
(t,\theta)\mapsto\BFt{u(\cdot,t)}=\tfn{\cdot}{t}\in\bigcap_{j=0,1,2}H^j\bigl((1,\infty)\times\RDucell,H^{2-j}(\torus)\bigr)
\]
while, using the definition of $\Tlap$ and the $\Rlattice$-periodicity of $V$, 
\begin{align*}
(\dt^2-\Tlap)\tfn{t}{\wt{x}}
&=\sum_{l\in\Rlattice} \re^{-\ri\theta.(\wt{x}+l)}\bigl(\dt^2-(-\ri\nabla_{\wt{x}}+\theta-\theta)^2+E\bigr)u(\wt{x}+l,t)\\
&=\sum_{l\in\Rlattice} \re^{-\ri\theta.(\wt{x}+l)}(\Delta+E)u(\wt{x}+l,t)\\
&=V(\wt{x},t)\sum_{l\in\Rlattice} \re^{-\ri\theta.(\wt{x}+l)}u(\wt{x}+l,t)
=V_t(\wt{x})\tfn{t}{\wt{x}}
\end{align*}
as elements of $L^2\bigl((1,\infty)\times\RDucell,L^2(\torus)\bigr)$.
Fubini's theorem then implies 
\[
\tfn{}{}\in\bigcap_{j=0,1,2}H^j\bigl((1,\infty),H^{2-j}(\torus)\bigr)\subset\fs{2},
\]
with $(\dt^2-\Tlap)\tfn{}{}=V_t\tfn{}{}$ as elements of $\fs{0}$, for almost all $\theta\in\RDucell$.

Since $\BFt{}$ is unitary and $\kappa\ge0$ \eqref{eq:hypestmainresper} gives
\[
\int_{\RDucell}\int_{0}^\infty\int_{\torus} \re^{2\lambda t}\abs{\tfn{t}{\wt{x}}}^2\rd\wt{x}\,\rd t\,\rd\theta
=\int_{\halfsp}\re^{2\lambda t}\abs{u(\wt{x},t)}^2\rd\wt{x}\,\rd t<+\infty.
\]
This leads to $\eqref{eq:hypestlemper}$ for almost all $\theta\in\RDucell$.

\medskip

Now suppose \eqref{eq:hypestmainresper} holds for some $\kappa>(d-1)/2$. 
Then $C_{2}=\sum_{l\in\Rlattice}{\bk{l}^{-2\kappa}}<+\infty$.
With $w=\sup\{\abs{\wt{x}}:\wt{x}\in\Rucell\}$ we also have
\[
\wt{x}\in\eRucell+l\implies\abs{l}\le\abs{\wt{x}}+\tfrac32 w
\implies\bk{l}^{2\kappa}\le C_{3}\bk{\wt{x}}^{2\kappa}
\]
where $C_{3}=(\max\{1+9w^2/2,2\})^\kappa$. Let $\theta\in\RDucell$.
For any $s\ge1$ \eqref{eq:ellregu1} then gives
\begin{align*}
\norm{\tfn{}{}}_{H^2(\torus\times(s,s+1))}
&\le\sum_{l\in\Rlattice}\,\norm{\re^{-\ri\theta.\wt{x}}}_{C^2(\Rucell+l)}\norm{u}_{H^2(\tcell{l}{s})}\\
&\le C_{1}C_{4}\sum_{l\in\Rlattice}\,\norm{u}_{L^2(\etcell{l}{s})}
\le C_{1}C_{4}(C_{2}N_{s})^{1/2},
\end{align*}
where $C_4=\sup_{\theta'\in\RDucell}\norm{\re^{-\ri\theta'.\wt{x}}}_{C^2(\R^{d-1})}$ and
\begin{align*}
N_{s}
=\sum_{l\in\Rlattice}\bk{l}^{2\kappa}\norm{u}_{L^2(\etcell{l}{s})}^2
&\le C_{3}\sum_{l\in\Rlattice}\int_{s-1}^{s+2}\int_{\eRucell+l}\bk{\wt{x}}^{2\kappa}\abs{u(\wt{x},t)}^2\rd\wt{x}\,\rd t\\
&=2^{d-1}C_{3}\int_{s-1}^{s+2}\int_{\R^{d-1}}\bk{\wt{x}}^{2\kappa}\abs{u(\wt{x},t)}^2\rd\wt{x}\,\rd t.
\end{align*}
Now $N_{s}<+\infty$ by \eqref{eq:hypestmainresper}, so $\tfn{}{}\in\fs{2}$.
A simpler version of this argument also gives 
\begin{align*}
\int_{1}^\infty\int_{\torus} \re^{2\lambda t}\abs{\tfn{t}{\wt{x}}}^2\rd\wt{x}\,\rd t
&\le C_{2}\sum_{s\in\N}\re^{2\lambda(s+1)}\sum_{l\in\Rlattice}\bk{l}^{2\kappa}\norm{u}_{L^2(\tcell{l}{s})}^2\\
&\le C_{2}C_{3}\re^{2\lambda}\int_{1}^\infty\int_{\R^{d-1}}\bk{\wt{x}}^{2\kappa}\re^{2\lambda t}\abs{u(\wt{x},t)}^2\rd\wt{x}\,\rd t
<+\infty.
\end{align*}

For $l\in\Rlattice$ and $\theta,\theta'\in\RDucell$ set 
$\delta_l(\theta,\theta')=\norm{\re^{-\ri\theta.(\wt{x}+l)}-\re^{-\ri\theta'.(\wt{x}+l)}}_{C^2(\Rucell)}$.
Arguing as above, 
\[
\norm{\tfn{}{}-\tfn[\theta']{}{}}_{H^2(\torus\times(s,s+1))}^2
\le C_{1}^2C_{2}\sum_{l\in\Rlattice}\delta_l(\theta,\theta')^2\,\bk{l}^{2\kappa}\norm{u}_{L^2(\etcell{l}{s})}^2.
\]
For fixed $l\in\Rlattice$, $\delta_l(\theta,\theta')\le 2C_{4}$
and $\delta_l(\theta,\theta')\to0$ as $\abs{\theta-\theta'}\to0$.
Since $N_{s}<+\infty$
dominated convergence then gives $\norm{\tfn{}{}-\tfn[\theta']{}{}}_{H^2(\torus\times(s,s+1))}\to0$
as $\abs{\theta-\theta'}\to0$. 
\end{proof}

It is straightforward to check that
\[
\spec(\Tlap)=\bigl\{\abs{k+\theta}^2-E:k\in\RDlattice\bigr\}.
\]
The next result establishes a key part of the hypothesis in theorem \ref{thm:mainresgen}.

\begin{proposition}
\label{prop:specgapTlap}
Let $\theta\in\RDucell$.
If $d\ge3$ suppose that $\Rlattice$ is a rational lattice (in $\R^{d-1}$) while $\theta$ has rational 
coordinates (with respect to a basis of $\RDlattice$). 
\begin{itemize}
\item[\upshape (i)]
If $d=2$ or $3$ then $\spec(\Tlap)$ contains arbitrarily large positive gaps.
\item[\upshape (ii)]
If $d\ge 4$ then there exists $\delta>0$ such that $\spec(\Tlap)$ contains 
infinitely many positive gaps of length at least $\delta$.
\end{itemize}
\end{proposition}

When $d=3$ we need non-trivial information about the gaps in the values realised by a binary quadratic form.
This is taken from \cite{P} and was previously observed in \cite{M} 
for the special case $\Rlattice=(2\pi\Z)^2$.

\begin{proof}
If $d=2$ we have $\RDlattice=f\Z$ and $\theta\in[0,f)$ for some $f\in\R_+$. 
Then $\spec(\Tlap)=\{(mf+\theta)^2-E:m\in\Z\}$; the existence of arbitrarily large gaps follows easily.

\medskip

Now suppose $d\ge3$. Choose a basis $\{\bdf_1,\dots,\bdf_{d-1}\}$ for the lattice $\RDlattice\subset\R^{d-1}$
corresponding to the unit cell $\RDucell$. If $k\in\RDlattice$ and $\theta\in\RDucell$ we can write
\begin{equation}
\label{kthbasisdecomp:eq}
k=m_1\bdf_1+\dots+m_{d-1}\bdf_{d-1}
\quad\text{and}\quad
\theta=\mu_1\bdf_1+\dots+\mu_{d-1}\bdf_{d-1}
\end{equation}
for some $m_i\in\Z$ and $\mu_i\in[0,1)$, $i=1,\dots,d-1$.
Since $\Rlattice$ and hence $\RDlattice$ are rational, 
we can find $\sigma>0$ and a positive definite integral quadratic form $\iqf$ so that
\[
\abs{k+\theta}^2=\sigma\iqf(m_1+\mu_1,\dots,m_{d-1}+\mu_{d-1})
\]
when $k\in\RDlattice$ and $\theta\in\RDucell$ are as given in \eqref{kthbasisdecomp:eq}.
Let $\Viqf=\{\iqf(\mathbf{m}):\mathbf{m}\in\Z^{d-1}\}\subseteq\Z$
denote the values of $\iqf$ realised by integer arguments.

Now suppose $\theta=\mu_1\bdf_1+\dots+\mu_{d-1}\bdf_{d-1}\in\RDucell$ has rational coefficients.
Thus we can write $\mu_i=r_i/l$ for some $l\in\N$ and $r_i\in\{0,\dots,l-1\}$, $i=1,\dots,d-1$.
Then
\[
\abs{k+\theta}^2=\frac{\sigma}{l^2}\,\iqf\bigl(lm_1+r_1,\dots,lm_{d-1}+r_{d-1}\bigr)\in\frac{\sigma}{l^2}\,\Viqf.
\]
Hence $\spec(\Tlap)+E\subseteq\sigma\Viqf/l^2\subseteq\sigma\Z/l^2$; part (ii) follows.

If $d=3$ then $\iqf$ is a positive definite binary quadratic form. 
By \cite[theorem 2]{P} there exists $C_{5}>0$ such that
\[
\#\{n\in\Viqf:n\le N\}\sim\frac{C_{5}N}{\sqrt{\log N}}
\quad\text{as $N\to+\infty$.}
\]
The existence of arbitrarily large positive gaps in $\Viqf$, and thus $\spec(\Tlap)$, follows.
\end{proof}

Theorem \ref{thm:mainresper} is now a straightforward corollary of theorem \ref{thm:mainresgen}.

\begin{proof}[Proof of theorem \ref{thm:mainresper}]
Suppose $d\ge3$ and \eqref{eq:hypestmainresper} holds for some $\kappa>(d-1)/2$ and all $\lambda>0$.
If $\theta\in\RDucell$  lemma \ref{lem:basicpropphitheta} shows $\tfn{}{}\in\fs{2}$ satisfies 
$(\dt^2-\Tlap)\tfn{}{}=V_t\tfn{}{}$ and
\[
\int_{1}^\infty \re^{2\lambda t}\norm{\tfn{t}{\cdot}}_{L^2(\torus)}^2\rd t<+\infty
\]
for all $\lambda>0$. 
If $\RDlattice$ is rational and $\theta\in\RDucell$ has rational coordinates with respect to $\RDlattice$, 
proposition \ref{prop:specgapTlap} and (a translated version of) theorem \ref{thm:mainresgen} then give 
$\tfn{t}{\cdot}\equiv0$ for all $t>1$. 
However the set of $\theta\in\RDucell$ with rational coordinates is 
dense, while $\tfn{}{}$ depends continuously on $\theta$ by lemma \ref{lem:basicpropphitheta}.
It follows that $\tfn{}{}\equiv0$ on $\torus\times(1,\infty)$ for all $\theta\in\RDucell$. 

Now let $\wt{x}\in\R^{d-1}$ and choose $\wt{x}'\in\Rucell$ and $l'\in\Rlattice$ with $\wt{x}=\wt{x}'+l'$.
Then, for $t>1$,
\begin{align}
0&=\frac{1}{\abs{\RDucell}}\int_{\RDucell}\re^{\ri\theta.\wt{x}}\tfn{t}{\wt{x}'}\rd\theta\nonumber\\
\label{invBFtrans:eq}
&=\frac{1}{\abs{\RDucell}}\int_{\RDucell}\sum_{l\in\Rlattice}\re^{\ri\theta.(l'-l)}u(\wt{x}'+l,t)\rd\theta
=\sum_{l\in\Rlattice}\delta_{ll'}u(\wt{x}'+l,t)=u(\wt{x},t).
\end{align}
Thus $u\equiv0$ on $\R^{d-1}\times(1,\infty)\subset\halfsp$.
Unique continuation (see \cite[theorem XIII.63]{RS4}, for example)
then shows $u$ is trivial on $\halfsp$.

The case $d=2$ can be handled similarly; in this case we get $\tfn{}{}\equiv0$ on $\torus\times(1,\infty)$ 
for almost all $\theta\in\RDucell$; this is sufficient to allow the reconstruction of $u$ as in \eqref{invBFtrans:eq}.
\end{proof}

\section{General result}
\label{sec:genres}

The Carleman type estimates we use for 
theorems \ref{thm:mainresarbAest} and \ref{thm:mainresgen} 
are stated in propositions \ref{prop:carlest43} and \ref{prop:carlest} respectively;
their proofs are deferred to \S\ref{sec:carlestpf}.
For convenience choose $\alpha\ge0$ with $\OpA+\alpha I\ge0$; in particular $\spec(\OpA)\subseteq[-\alpha,+\infty)$. 

\begin{proposition}
\label{prop:carlest43}
Let $\phi\in C^2_0(\R_+,\hs[2])$ and choose $\epsilon>0$ so that $\phi(t)=0$ for $t<\epsilon$. 
If $\gcen\ge\epsilon^{-4/3}$ then
\begin{equation}
\label{eq:carlest43}
\lambda^3\int_0^\infty \re^{2\gcen t^{4/3}}\norm{\phi(t)}_{\hs}^2\rd t
\le\int_0^\infty \re^{2\gcen t^{4/3}}\norm{(\OpH)\phi(t)}_{\hs}^2\rd t.
\end{equation}
\end{proposition}

\begin{proposition}
\label{prop:carlest}
Suppose $(\glb^2,\gub^2)\cap\spec(\OpA)=\emptyset$ for some $0<\glb<\gub$ with $3\glb^2>\alpha$.
Set $\gcen=(\glb+\gub)/2$ and $\glen=\gub-\glb$. 
For any $\phi\in C^2_0(\R_+,\hs[2])$ we have
\begin{equation}
\label{eq:carlest}
\frac{\glb^2\glen^2}{4}\int_0^\infty \re^{2\gcen t}\norm{\phi(t)}_{\hs}^2\rd t
\le\int_0^\infty \re^{2\gcen t}\norm{(\OpH)\phi(t)}_{\hs}^2\rd t.
\end{equation}
\end{proposition}

To apply these Carleman estimates we need to use the 
bounds on $\norm{\phi(t)}_{\hs}$ given by \eqref{eq:hypestmainresarbAest} or \eqref{eq:hypestmainresgen} 
to obtain similar bounds for $\norm{\dt\phi(t)}_{\hs}$; 
this can be done using the `elliptic regularity' of the operator $\OpH$.

\begin{lemma}
\label{lem:ellreg}
Let $\phi\in\fs{2}$ and suppose $\norm{(\OpH)\phi(t)}_{\hs}\le\beta\norm{\phi(t)}_{\hs}$
for some $\beta>0$ and all $t>0$. Let $\sigma\ge1$.
If
\[
\int_0^\infty \re^{2\lambda t^\sigma}\norm{\phi(t)}_{\hs}^2\rd t<+\infty
\]
for all $\lambda>0$ then
\[
\int_{\epsilon}^\infty \re^{2\lambda t^\sigma}\norm{\dt\phi(t)}_{\hs}^2\rd t<+\infty
\]
for all $\epsilon,\lambda>0$.
\end{lemma}

\begin{proof}
Set $\OpN=(\OpA+\alpha I)^{1/2}$. 
For any $\ctf\in C^\infty_0(\R_+)$ we have
\begin{align*}
&\int_0^\infty\ctf^2\bigl[\norm{\dt\phi}_{\hs}^2+\norm{\OpN\phi}_{\hs}^2\bigr]
=-\int_0^\infty\Real\bigl[\ctf^2\bigipd{\phi}{(\dt^2-\OpN^2)\phi}_{\hs}+2\ctf\ctf'\ipd{\phi}{\dt\phi}_{\hs}\bigr]\\
&\qquad{}\le\int_0^\infty\ctf^2\bigabs{\bigipd{\phi}{(\alpha I-(\OpH))\phi}_{\hs}}
+2\int_0^\infty\bigabs{\ctf\ctf'\ipd{\phi}{\dt\phi}_{\hs}}\\
&\qquad{}\le(\alpha+\beta)\int_0^\infty\ctf^2\norm{\phi}_{\hs}^2
+2\int_0^\infty(\ctf')^2\norm{\phi}_{\hs}^2
+\frac12\int_0^\infty\ctf^2\norm{\dt\phi}_{\hs}^2
\end{align*}
since $4\abs{hh'\ipd{\phi}{\dt\phi}_{\hs}}\le4(h')^2\norm{\phi}_{\hs}^2+h^2\norm{\dt\phi}_{\hs}^2$.
If $\epsilon>0$ we can then choose $\ctf$ to be an appropriate translate of a 
function with support in $(-\epsilon,1+\epsilon)$ and value $1$ on $(0,1)$
to find $C_{6}>0$ such that
\[
\int_{s}^{s+1}\norm{\dt\phi}_{\hs}^2
\le C_{6}\int_{s-\epsilon}^{s+1+\epsilon}\norm{\phi}_{\hs}^2
\]
for all $s\ge\epsilon$.
If $\lambda'>\lambda>0$ we can find $C_{7}>0$ with 
$\gcen(n+1+\epsilon)^\sigma-\gcen' n^\sigma\le C_{7}$
for all $n\ge0$. Then
\begin{align*}
\int_{\epsilon}^\infty \re^{2\gcen t^\sigma}\norm{\dt\phi}_{\hs}^2
&\le\sum_{n=0}^\infty \re^{2\gcen(n+1+\epsilon)^\sigma}\int_{n+\epsilon}^{n+1+\epsilon}\norm{\dt\phi}_{\hs}^2\\
&\le C_{6}\re^{2C_{7}}\sum_{n=0}^\infty \re^{2\gcen'n^\sigma}
\int_{n}^{n+1+2\epsilon}\norm{\phi}_{\hs}^2\\
&\le C_{6}\re^{2C_{7}}\sum_{n=0}^\infty 
\int_{n}^{n+1+2\epsilon}\re^{2\gcen't^\sigma}\norm{\phi}_{\hs}^2\\
&\le 2(1+\epsilon)C_{6}\re^{2C_{7}}\int_0^\infty \re^{2\gcen't^\sigma}\norm{\phi}_{\hs}^2.
\end{align*}
The result follows.
\end{proof}

\begin{proof}[Proof of theorem \ref{thm:mainresarbAest}]
Set $\psi=(\OpH)\phi\in\fs{0}$. Let $\epsilon\in(0,1)$ 
and suppose $\gcen>(2\beta^2)^{1/3},\,(\epsilon/2)^{-4/3}$.
For each $R>1$ choose a cut-off function $\ctf_{R}\in C^\infty_0(\R_+)$ with $\supp(\ctf_{R})\subseteq(\epsilon/2,2R)$,
$\ctf_{R}(t)=1$ for $t\in[\epsilon,R]$, $\ctf_{R}(t)$ independent of $R$ for $t\le\epsilon$,
\begin{equation}
\label{eq:Linftybndctfderiv43}
\sup_{t\in[R,2R]}\ctf_{R}'(t)\le\frac2{R}
\quad\text{and}\quad
\sup_{t\in[R,2R]}\ctf_{R}''(t)\le\frac8{R^2}.
\end{equation}
By using a mollifier (for example) we can find an approximating sequence for $\ctf_{R}\phi$
in $C^2_0(\R^+,\hs[2])$; we may further assume elements of this sequence are supported in $[\epsilon/2,\infty)$.  
Apply proposition \ref{prop:carlest43} to elements of this sequence;
taking the limit and noting that $(\OpH)(\ctf_{R}\phi)=\ctf_{R}''\phi+2\ctf_{R}'\dt\phi+\ctf_{R}\psi$, 
we get the estimate
\begin{align*}
&\gcen^3\int_0^\infty \re^{2\gcen t^{4/3}}\ctf_{R}^2\norm{\phi}_{\hs}^2
\le\int_0^\infty \re^{2\gcen t^{4/3}}\bignorm{\ctf_{R}''\phi+2\ctf_{R}'\dt\phi+\ctf_{R}\psi}_{\hs}^2\\
&\qquad{}\le2\beta^2\int_0^\infty \re^{2\gcen t^{4/3}}\ctf_{R}^2\norm{\phi}_{\hs}^2
+\int_0^\infty \re^{2\gcen t^{4/3}}\bigl[4(\ctf_{R}'')^2\norm{\phi}_{\hs}^2+16(\ctf_{R}')^2\norm{\dt\phi}_{\hs}^2\bigr].
\end{align*}
Rearranging and using \eqref{eq:Linftybndctfderiv43} we get
\[
(\gcen^3-2\beta^2)\int_\epsilon^R \re^{2\gcen t^{4/3}}\norm{\phi}_{\hs}^2
\le C_{8}\re^{2\gcen\epsilon^{4/3}}+\remRphi
\]
where
\[
C_{8}=\int_0^\epsilon\bigl[4(\ctf_{R}'')^2\norm{\phi}_{\hs}^2+16(\ctf_{R}')^2\norm{\dt\phi}_{\hs}^2\bigr],
\]
which is independent of $R$ and $\gcen$, and
\[
\remRphi=\frac{2^8}{R^4}\int_{R}^{2R}\re^{2\gcen t^{4/3}}\norm{\phi}_{\hs}^2
+\frac{2^6}{R^2}\int_{R}^{2R}\re^{2\gcen t^{4/3}}\norm{\dt\phi}_{\hs}^2.
\]
However, our hypothesis and lemma \ref{lem:ellreg}(i) give 
\[
\int_0^\infty \re^{2\gcen t^{4/3}}\norm{\phi}_{\hs}^2,\ 
\int_{\epsilon/2}^\infty \re^{2\gcen t^{4/3}}\norm{\dt\phi}_{\hs}^2<+\infty.
\]
Thus $\remRphi\to0$ as $R\to\infty$. It follows that
\[
(\gcen^3-2\beta^2)\int_\epsilon^\infty \re^{2\gcen t^{4/3}}\norm{\phi}_{\hs}^2
\le C_{8}\re^{2\gcen\epsilon^{4/3}}.
\]
For any $\epsilon_1>\epsilon$ we then get
\begin{equation}
\label{eq:egceD}
(\gcen^3-2\beta^2)\int_{\epsilon_1}^\infty\norm{\phi}_{\hs}^2\le C_{8}\re^{-2\gcen(\epsilon_1^{4/3}-\epsilon^{4/3})}.
\end{equation}
However this inequality will
be contradicted for sufficiently large $\gcen$ if $\int_{\epsilon_1}^\infty\norm{\phi}_{\hs}^2>0$.
Hence $\phi(t)=0$ for $t>\epsilon$. 
Since $\epsilon\in(0,1)$ was arbitrary the result follows.
\end{proof}

\begin{proof}[Proof of theorem \ref{thm:mainresgen}]
Define a non-increasing function by $\blsup(t)=\sup_{t'> t}b(t')$ for $t\ge0$. 
Suppose $(\glb^2,\gub^2)\cap\spec(\OpA)=\emptyset$ for some $0<\glb<\gub$ with $3\glb^2>\alpha$. 
Set $\gcen=(\glb+\gub)/2$ and $\glen=\gub-\glb$. 
Suppose $\glb^2\glen^2>8\blsup^2(t_0)$ for some $t_0\ge0$. 
We can now emulate the proof of theorem \ref{thm:mainresarbAest}; following the argument 
as far as \eqref{eq:egceD} we get
\begin{equation}
\label{eq:fincarlestbnd}
\Bigl(\frac{\glb^2\glen^2}{4}-2\blsup^2(t_0)\Bigr)\int_{t_2}^\infty\norm{\phi}_{\hs}^2
\le C_{9}\re^{-2\gcen(t_2-t_1)}
\end{equation}
for any $t_2>t_1>t_0$, where $C_{9}$ is a constant which is independent of $\glb,\gub,\glen$ and $\gcen$. 

Our hypothesis gives a sequence of disjoint intervals $(\glb[n]^2,\gub[n]^2)$, $n=1,2,\dots$, 
in $\R_+\setminus\spec(\OpA)$ with either (i) $\gub[n]^2-\glb[n]^2\to\infty$ as $n\to\infty$, or (ii)
$\gub[n]^2-\glb[n]^2\ge\delta$ for all $n$ and $\blsup(t)\to0$ as $t\to\infty$.
We may further assume $\glb[n]$ is increasing and $\gub[n]\le3\glb[n]$ for all $n$. 
Then $\gcen[n]=(\gub[n]+\glb[n])/2\to\infty$ as $n\to\infty$, while 
\[
\gub[n]^2-\glb[n]^2=(\gub[n]+\glb[n])(\gub[n]-\glb[n])
\le4\glb[n](\gub[n]-\glb[n])=4\glb[n]\glen[n]
\]
for all $n$. We complete the argument for the two cases separately.

\medskip

\noindent
(i) In this case $\glb[n]\glen[n]\to\infty$ as $n\to\infty$. 
Taking $t_0=0$ we will contradict \eqref{eq:fincarlestbnd} for sufficiently large $n$
unless $\int_{t_2}^\infty\norm{\phi}_{\hs}^2=0$ for all $t_2>0$. Hence $\phi\equiv0$ on $\R_+$.
\medskip

\noindent
(ii) Choose $t_0$ so $\blsup^2(t_0)<2^{-7}\delta^2$. Then $\glb[n]^2\glen[n]^2/4\ge2^{-6}\delta^2>2\blsup^2(t_0)$ 
for all $n$. Since $\gcen[n]\to\infty$ \eqref{eq:fincarlestbnd} will be contradicted for sufficiently large $n$
unless $\int_{t_2}^\infty\norm{\phi}_{\hs}^2=0$ for all $t_2>t_0$.
Thus $\phi(t)=0$ for $t>t_0$. However \eqref{eq:hypestmainresgen} 
then implies \eqref{eq:hypestmainresarbAest},
so $\phi\equiv0$ on $\R_+$ by theorem \ref{thm:mainresarbAest}
\end{proof}

\section{Carleman estimates}
\label{sec:carlestpf}

\begin{proof}[Proof of proposition \ref{prop:carlest43}]
Define $\Ftew,\ftew:\R_+\to\R$ by $\Ftew(t)=\lambda t^{4/3}$ and 
$\ftew=\Ftew'$. Then
\[
\re^\Ftew(\OpH)(\re^{-\Ftew}\,\cdot\,)=\dt^2-\OpAw-\OpL
\]
where $\OpAw=\OpA-\ftew^2$ and $\OpL=2\ftew\dt+\ftew'$.
If $\psi\in C^2_0(\R_+,\hs[2])$ then 
\begin{gather*}
\norm{\OpL\psi}_{\hs}^2
=4\ftew^2\norm{\dt\psi}_{\hs}^2+(\ftew^2)'\dt\norm{\psi}_{\hs}^2+(\ftew')^2\norm{\psi}_{\hs}^2,\\
2\Real\ipd{\OpL\psi}{\dt^2\psi}_{\hs}
=2\ftew\dt\norm{\dt\psi}_{\hs}^2+\ftew'\bigl(\dt^2\norm{\psi}_{\hs}^2-2\norm{\dt\psi}_{\hs}^2\bigr)
\end{gather*}
and
\begin{align*}
2\Real\ipd{\OpL\psi}{\OpAw\psi}_{\hs}
&=2\ftew\dt\ipd{\psi}{\OpA\psi}_{\hs}+2\ftew'\ipd{\psi}{\OpA\psi}_{\hs}
-2\ftew^3\dt\norm{\psi}_{\hs}^2-2\ftew^2\ftew'\norm{\psi}_{\hs}^2\\
&=2\dt\bigl(\ftew\ipd{\psi}{\OpA\psi}_{\hs}\bigr)-2\dt\bigl(\ftew^3\norm{\psi}_{\hs}^2\bigr)
+\tfrac43(\ftew^3)'\norm{\psi}_{\hs}^2.
\end{align*}
Integration then leads to
\begin{align*}
&\int_0^\infty\bignorm{\re^\Ftew(\OpH)(\re^{-\Ftew}\psi)}_{\hs}^2
\ge\int_0^\infty\!\bigl(\norm{\OpL\psi}_{\hs}^2-2\Real\ipd{\OpL\psi}{\dt^2\psi}_{\hs}+2\Real\ipd{\OpL\psi}{\OpAw\psi}_{\hs}\bigr)\\
&\qquad=\int_0^\infty4(\ftew^2+\ftew')\norm{\dt\psi}_{\hs}^2
+\int_0^\infty\Bigl(\frac43(\ftew^3)'-(\ftew')^2-2\ftew\ftew''-\ftew'''\Bigr)\norm{\psi}_{\hs}^2.
\end{align*}
However $\ftew'\ge0$, $(\ftew^3)'=(4\lambda/3)^3\ge3\lambda^3/4$ and
\[
-\bigr((\ftew')^2+2\ftew\ftew''+\ftew'''\bigr)(t)
=\frac8{81}\lambda t^{-4/3}(6\lambda-5t^{-4/3}),
\]
which is positive when $\lambda\ge t^{-4/3}$. 
Taking $\psi=\re^{\Ftew}\phi$ now completes the result.
\end{proof}

\begin{proof}[Proof of proposition \ref{prop:carlest}]
Let $\proj{-}$ and $\proj{+}$ denote the orthogonal spectral projections of $\OpA$ on $\hs$ 
corresponding to the intervals $[-\alpha,\glb^2]$ and $[\gub^2,+\infty)$ respectively.
Note that $\proj{-}+\proj{+}=I$.
Set $\psi=(\OpH)\phi$.
Denote the corresponding projections by 
$\phi_\pm=\proj{\pm}\phi$ and $\psi_\pm=\proj{\pm}\psi$;
in particular, $\psi_\pm=(\OpH)\phi_\pm$. 
Introduce the operator $\OpN=(\OpA\proj{+})^{1/2}$; this commutes with $\OpA$, $\proj{\pm}$ and $\dt$,
and defines bounded maps $\hs[2]\to\hs[1]$ and $\hs[1]\to\hs$.

To move from a second order equation to a system of first order ones 
we introduce spaces $\HS[j]=\hs[j]\otimes\C^2$ for $j=0,1,2$, and put $\HS=\HS[0]$.
Setting
\[
\Qroj{0}=\begin{pmatrix}\proj{-}&0\\0&\proj{-}\end{pmatrix},\quad
\Qroj{1}=\begin{pmatrix}\proj{+}&0\\0&0\end{pmatrix}
\quad\text{and}\quad
\Qroj{2}=\begin{pmatrix}0&0\\0&\proj{+}\end{pmatrix}
\]
gives orthogonal projections on $\HS$ with $\Qroj{0}+\Qroj{1}+\Qroj{2}=\mathbf{I}$. Let
\[
\bPhi=\re^{\gcen t}\begin{pmatrix}\dt\phi+(\glb\proj{-}+\OpN)\phi\\
\dt\phi-(\glb\proj{-}+\OpN)\phi\end{pmatrix}
\quad\text{and}\quad
\bPsi=\re^{\gcen t}\begin{pmatrix}\psi\\\psi\end{pmatrix}.
\]
For $j=0,1,2$
denote the corresponding projections by
$\bPhi_j=\Qroj{j}\bPhi$ and $\bPsi_j=\Qroj{j}\bPsi$.
Then
\[
\bPhi_{0}=\re^{\gcen t}\begin{pmatrix}(\dt+\glb)\phi_{-}\\(\dt-\glb)\phi_{-}\end{pmatrix},\quad
\bPhi_{1}=\re^{\gcen t}\begin{pmatrix}(\dt+\OpN)\phi_{+}\\0\end{pmatrix}
\]
and
\[
\bPhi_2=\re^{\gcen t}\begin{pmatrix}0\\(\dt-\OpN)\phi_{+}\end{pmatrix}.
\]
Now $(\dt\pm\glb)\phi_{-}\in C^1_0(\R_+,\hs[2])$ 
and $(\dt\pm\OpN)\phi_{+}\in C^1_0(\R_+,\hs[1])$ so
$\bPhi_0\in C^1_0(\R_+,\HS[2])$ and $\bPhi_1,\bPhi_2\in C^1_0(\R_+,\HS[1])$, 
while $\bPsi_0,\bPsi_1,\bPsi_2\in C^0_0(\R_+,\HS)$.
We also have
\[
\dt(\dt\pm\glb)\phi_{-}
=(\OpA\pm\glb\dt)\phi_{-}+\psi_{-}
\quad\text{and}\quad
\dt(\dt\pm\OpN)\phi_{+}
=\pm\OpN(\dt\pm\OpN)\phi_{+}+\psi_{+}.
\]
It follows that $\dt\bPhi_j=\OpBM{j}\bPhi_j+\bPsi_j$ for $j=0,1,2$, where 
\begin{gather*}
\OpBM{0}=\frac1{2\glb}\begin{pmatrix}\OpA+\glb^2&-\OpA+\glb^2\\\OpA-\glb^2&-\OpA-\glb^2\end{pmatrix}\Qroj{0}+\gcen\Qroj{0},\\[5pt]
\OpBM{1}=\begin{pmatrix}\OpN&0\\0&0\end{pmatrix}+\gcen\Qroj{1}
\quad\text{and}\quad
\OpBM{2}=\begin{pmatrix}0&0\\0&-\OpN\end{pmatrix}+\gcen\Qroj{2}.
\end{gather*}
Therefore
\begin{align}
\dt\norm{\bPhi_j}_{\HS}^2
&=\ipd{\dt\bPhi_j}{\bPhi_j}_{\HS}+\ipd{\bPhi_j}{\dt\bPhi_j}_{\HS}
\nonumber\\
\label{eq:diffnorm}
&=\bigipd{\bPhi_j}{(\OpBM{j}^*+\OpBM{j})\bPhi_j}_{\HS}
+2\Real\ipd{\bPsi_j}{\bPhi_j}_{\HS}.
\end{align}

Now $-3\glb\proj{-}\le-\alpha\glb^{-1}\proj{-}\le\glb^{-1}\OpA\proj{-}\le\glb\proj{-}$ 
and $\OpN\ge\gub\proj{+}$ so
\begin{align*}
\OpBM{0}^*+\OpBM{0}
&=\begin{pmatrix}(\glb^{-1}\OpA+\glb+2\gcen)\proj{-}&0\\0&(-\glb^{-1}\OpA-\glb+2\gcen)\proj{-}\end{pmatrix}\\
&\ge2(\gcen-\glb)\begin{pmatrix}\proj{-}&0\\0&\proj{-}\end{pmatrix}
=\glen\Qroj{0},
\end{align*}
while
\[
\OpBM{1}^*+\OpBM{1}=\begin{pmatrix}2\OpN+2\gcen\proj{+}&0\\0&0\end{pmatrix}
\ge2(\gub+\gcen)\begin{pmatrix}\proj{+}&0\\0&0\end{pmatrix}
\ge\glen\Qroj{1}.
\]
For $j=0,1$ \eqref{eq:diffnorm} then leads to
\[
\dt\norm{\bPhi_j}_{\HS}^2
\ge\glen\norm{\bPhi_j}_{\HS}^2-2\norm{\bPsi_j}_{\HS}\norm{\bPhi_j}_{\HS}
\ge\frac{\glen}2\norm{\bPhi_j}_{\HS}^2-\frac2{\glen}\norm{\bPsi_j}_{\HS}^2.
\]
Since $\norm{\bPhi_j}_{\HS}^2\in C^1_0(\R_+)$ we can now integrate this inequality to get
\begin{equation}
\label{eq:TU1ineq}
\frac{\glen^2}4\int_0^\infty\norm{\bPhi_j}_{\HS}^2\le\int_0^\infty\norm{\bPsi_j}_{\HS}^2
\end{equation}
for $j=0,1$.
A simpler version of the above argument gives
\[
\OpBM{2}^*+\OpBM{2}=\begin{pmatrix}0&0\\0&-2\OpN+2\gcen\proj{+}\end{pmatrix}
\le2(\gcen-\gub)\begin{pmatrix}0&0\\0&\proj{+}\end{pmatrix}
=-\glen\Qroj{2},
\]
so
\[
\dt\norm{\bPhi_2}_{\HS}^2
\le-\glen\norm{\bPhi_2}_{\HS}^2+2\norm{\bPsi_2}_{\HS}\norm{\bPhi_2}_{\HS}
\le-\frac{\glen}2\norm{\bPhi_2}_{\HS}^2+\frac2{\glen}\norm{\bPsi_2}_{\HS}^2,
\]
and hence \eqref{eq:TU1ineq} for $j=2$.
However
\begin{align*}
&\norm{\bPhi_0}_{\HS}^2+\norm{\bPhi_1}_{\HS}^2+\norm{\bPhi_2}_{\HS}^2\\
&\qquad{}=\re^{2\gcen t}\bigl(\bignorm{(\dt+\glb)\phi_-}_{\hs}^2+\bignorm{(\dt-\glb)\phi_-}_{\hs}^2\\
&\qquad\qquad{}+\bignorm{(\dt+\OpN)\phi_+}_{\hs}^2+\bignorm{(\dt-\OpN)\phi_+}_{\hs}^2\bigr)\\
&\qquad{}=2\re^{2\gcen t}\bigl(\norm{\dt\phi_-}_{\hs}^2+\glb^2\norm{\phi_-}_{\hs}^2+
\norm{\dt\phi_+}_{\hs}^2+\norm{\OpN\phi_+}_{\hs}^2\bigr)\\
&\qquad{}\ge2\re^{2\gcen t}\bigl(\norm{\dt\phi_-}_{\hs}^2+\glb^2\norm{\phi_-}_{\hs}^2+
\norm{\dt\phi_+}_{\hs}^2+\gub^2\norm{\phi_+}_{\hs}^2\bigr)\\
&\qquad{}\ge2\re^{2\gcen t}\bigl(\norm{\dt\phi}_{\hs}^2+\glb^2\norm{\phi}_{\hs}^2\bigr),
\end{align*}
while
\[
\norm{\bPsi_0}_{\HS}^2+\norm{\bPsi_1}_{\HS}^2+\norm{\bPsi_2}_{\HS}^2
=2\re^{2\gcen t}\bigl(\norm{\psi_-}_{\hs}^2+\norm{\psi_+}_{\hs}^2\bigr)
=2\re^{2\gcen t}\norm{\psi}_{\hs}^2.
\]
These can be combined with \eqref{eq:TU1ineq} for $j=0,1,2$ to complete the result.
\end{proof}

\section*{Acknowledgements}
The author wishes to acknowledge the hospitality of the Isaac Newton Institute for Mathematical Sciences in
Cambridge, UK, where this work was initiated during the programme \emph{Periodic and Ergodic Spectral Problems}.
The author also wishes to thank the referee for several useful comments and suggestions.

\label{lastpage}
\end{document}